\documentclass[a4paper]{article}

\usepackage[english]{babel}
\usepackage[utf8]{inputenc}
\usepackage{amssymb,amsmath,amsthm, latexsym}
\usepackage{graphicx}
\usepackage[colorinlistoftodos]{todonotes}
\usepackage{graphicx}
\usepackage{tikz}
\usepackage{enumitem}

\newtheorem{theorem}{Theorem}[section]
\newtheorem{proposition}[theorem]{Proposition}
\newtheorem{corollary}[theorem]{Corollary}
\newtheorem{lemma}[theorem]{Lemma}

\theoremstyle{definition}
\newtheorem{definition}[theorem]{Definition}

\newtheorem{remark}[theorem]{Remark}

\title{Invariance of the restricted $p$-power map on integrable derivations under stable equivalences}

\author{Lleonard Rubio y Degrassi}

\date{\today}

\begin{document}
\maketitle

\begin{abstract}
We show that the $p$-power maps in the first Hochschild cohomology space of finite-dimensional selfinjective algebras over a field of prime characteristic $p$ commute with stable equivalences of Morita type on the subgroup of classes represented by integrable derivations. We show, by giving an example, that the $p$-power maps do not necessarily commute with arbitrary transfer maps in the Hochschild cohomology of symmetric algebras.
\end{abstract}

\section{Introduction}

Let  $k$ be a field of prime characteristic $p$.
For symmetric $k$-algebras, it is shown  in $\cite{KLZ}$ that the Gerstenhaber bracket in Hochschild cohomology commutes with the transfer maps introduced in $\cite{Markus1}$. Zimmermann proved in $\cite{Zim}$ that the $p$-power map on (the positive part of) Hochschild cohomology commutes with derived equivalences. We show in this paper that the $p$-power map, restricted to the classes of integrable derivations, commutes with stable equivalences of Morita type between finite-dimensional  selfinjective algebras. We also show, by giving an example, that $p$-power maps need not commute with arbitrary  transfer maps in the Hochschild cohomology of symmetric algebras. 
To state our main result, we use the following notation: let $A$ be a finite-dimensional selfinjective $k$-algebra. For $r$ a positive integer, we denote by $\mathrm{Aut}_r(A[[t]])$ the subgroup of $k[[t]]$-algebra automorphism of $A[[t]]$ which induce the identity on $A[[t]]/t^rA[[t]]$. If $\alpha \in \mathrm{Aut}_r(A[[t]])$, then there is a unique $k[[t]]$-linear map $\mu$ on $A[[t]]$ such that $\alpha(a)= a+t^r\mu(a)$ for all $a \in A[[t]]$.\\
An easy verification (see Proposition $\ref{1.3}$) shows that the map $\bar{\mu}$ induced by $\mu$ on the quotient $A[[t]]/tA[[t]]\cong A$ is a derivation; any such derivation is called $\it{r}$-$\it {integrable}$. We denote by $\mathrm{HH}^1_r(A)$ the image in $\mathrm{HH}^1(A)$ of all $r$-integrable derivations. Let $A$, $B$ be finite-dimensional  selfinjective $k$-algebras,  $M$ be an $A$-$B$-bimodule and $N$ a $B$-$A$-bimodule. Following Brou\'e $\cite{B}$, we say that $M$ and $N$ induce a stable equivalence of Morita type between $A$ and $B$ if $M$, $N$ are finitely generated projective as left and right modules with the property that $M \otimes_B N \cong A \oplus X$ for some projective $A$-$A$-bimodule $X$ and $N \otimes_A M \cong B \oplus Y$ for some projective $B$-$B$-bimodule $Y$. If $A$, $B$ are symmetric then $N$ can by replaced by $M^{\vee}$.

\begin{theorem}
\label{mainthm}
Let $A, B$ be finite-dimensional selfinjective $k$-algebras with separable semisimple quotients, and let $M$, $N$ be an $A$-$B$-bimodule, $B$-$A$-bimodule, respectively, inducing a stable equivalence of Morita type between $A$ and $B$. For any positive integer $r$, the $p$-power map sends $\mathrm{HH}^1_r(A)$ to $\mathrm{HH}^1_{rp}(A)$, and we have a commutative diagram of maps 
 \begin{center}
 \begin{minipage}{0.3 \textwidth}
      \begin{tikzpicture}[node distance=3.5cm, auto]
  \node (HB1) {$\mathrm{HH}^1_{r}(A)$};
  \node (HB2) [below of=HB1][node distance=2 cm ] {$\mathrm{HH}^1_{rp}(A)$};
  \node (HA1) [right of=HB1] {$\mathrm{HH}^1_{r}(B)$};
  \node (HA2) [right of=HB2] {$\mathrm{HH}^1_{rp}(B)$};
  \draw[->] (HB1) to node {$\cong$} (HA1);
  \draw[->] (HA1) to node {$[p]$} (HA2);
 \draw[->] (HB1) to node   {$[p]$} (HB2);
\draw[->] (HB2) to node  {$\cong$} (HA2);
\end{tikzpicture}
\end{minipage}
\end{center}

where the horizontal isomorphisms are induced by the  functor $N\otimes_A\ -\ \otimes_A M$, and where the vertical maps are the $p$-power maps.
\end{theorem}
 In Section 2, we recall some basic results. In Section 3 we prove the main results concerning $r$-integrable derivation that allow us to prove in Section 4 the Theorem $\ref{mainthm}$. In the last section we  provide an example of when the $p$-power map  does not commute with a transfer map between the Hochschild cohomology of two symmetric algebras.

\section{Background}
 Let $A$ be a finite-dimensional algebra over $k$. For any integer $n\geq 0$ and any $A\otimes_k A^{op}$-module $M$ the Hochschild cohomology
 of degree $n$ of $A$ with coefficients in $M$  is denoted by
 $\mathrm{HH}^n(A;M)$ in particular $\mathrm{HH}^n(A)=\mathrm{HH}^n(A;A)$. It is well known
 that $\mathrm{HH}^0(A)=Z(A)$ and $\mathrm{HH^1}(A)$ is the space of derivations modulo inner
 derivations. The direct sum $\bigoplus_{n\geq 0} \mathrm{HH}^n(A)$ is a  Gerstenhaber algebra, in particular $\mathrm{HH}^1(A)$ is a Lie algebra. In addition, if the characteristic of $k$ is positive, there is a map $[p]: \mathrm{HH}^1(A) \to \mathrm{HH}^1(A)$, called $p$-power map. This map is induced by the map sending a derivation $f$ to $f^p$ that is $f$ composed $p$-times with itself. Then $\mathrm{HH}^1(A)$ endowed with the $p$-power becomes a restricted Lie algebra. 
   Let $A[[t]]$ be the formal power series with coefficients in $A$. By $\cite[2.1]{Markus}$  the canonical map $A[[t]] \to A[[t]]/ t^rA[[t]]$ induces an isomorphism 
   \begin{equation}
   \label{1eq}
   \mathrm{HH}^n(A[[t]];A[[t]]/t^rA[[t]])\cong \mathrm{HH}^n(A[[t]]/t^rA[[t]]).   
   \end{equation}
   for all $n \geq 0$ and $r > 0$.
The following is well known:
\begin{lemma}
\label{rmk1}
Let $A$ be a finite-dimensional algebra over $k$ and let $A[[t]]$ be the formal power series with coefficients in $A$. Then the multiplication in $A[[t]]$ induce a $k[[t]]$-algebra isomorphism $k[[t]] \otimes_k A \cong A[[t]]$.
\end{lemma}
\begin{proof} 
The isomorphism sends $\sum_{i\geq 0} \lambda_i t^i \otimes a$ to $\sum_{i \geq 0} \lambda_i a t^i$ where $\lambda_i \in k$ and $a \in A$. In order to show that this is an isomorphism, we construct its inverse as follows: let $\sum_{i\geq 0}a_it^i \in A[[t]]$ and let $\{e_j\}_{1 \leq j \leq n}$ be a $k$-basis of $A$. Write $a_i=\sum_{j=1}^n\mu_{ij}e_j$ for every non-negative integer $i$ where $\mu_{ij} \in k$. The inverse map sends $\sum_{i \geq 0} a_i  t^i$ to  $ \sum_{j=1}^n\Big(\sum_{i\geq 0}\mu_{ij}t^i \otimes e_j\Big)$.
   \end{proof}

\begin{corollary}
\label{cor1}
Let $A$ be a finite-dimensional algebra over $k$ and let $r$ be a positive integer. Then the canonical map $Z(A[[t]]) \to Z(A[[t]]/t^rA[[t]])$ is surjective.
\end{corollary}

Let $n$ be an integer. We recall that if 
  \begin{center}
 \begin{tikzpicture}[node distance=2 cm,auto]
  \node (P2) {$0$};
  \node (P1) [right of=P2] {$X$};
  \node (P0) [right of=P1] {$Y$};
  \node (A) [right of=P0] {$ Z$};
    \node (A0) [right of=A] {$0$};
 \draw[->] (P2) to node  {$$} (P1);
  \draw[->] (P1) to node {$\tau$} (P0);
 \draw[->] (P0) to node   {$\sigma$} (A);
\draw[->] (A) to node  {$$} (A0);
\end{tikzpicture}
\end{center}

 is a short exact sequence of cochain complexes with differentials $\delta, \epsilon, \zeta$ respectively, then this induces a long exact sequence
 
 \begin{center}
      \begin{tikzpicture}[node distance=2 cm,auto]
  \node (P2) {$\dots$};
  \node (P1) [right of=P2] {$\mathrm{H}^n(X)$};
  \node (P0) [right of=P1] {$\mathrm{H}^n(Y)$};
  \node (A) [right of=P0] {$ \mathrm{H}^n(Z)$};
    \node (A0) [right of=A] {$\mathrm{H}^{n+1}(X)$};
        \node (A1) [right of=A0] {$\dots$};
  \draw[->] (P2) to node  {$$} (P1);
  \draw[->] (P1) to node {$\mathrm{H}^n(\tau)$} (P0);
 \draw[->] (P0) to node   {$\mathrm{H}^n(\sigma)$} (A);
\draw[->] (A) to node  {$d^n$} (A0);
\draw[->] (A0) to node  {$$} (A1);
\end{tikzpicture}
\end{center}

  depending functorially on the short exact sequence, where $d^n$ is called the connecting homomorphism which is obtained in the following way: let $\bar{z}=z+\mathrm{Im}(\zeta^{n-1}) \in \mathrm{H}^n(Z)$
for some $z \in \mathrm{Ker}(\zeta^n)\subseteq Z^n$. Since $\sigma$ is surjective in each degree there is $y \in Y^n$ such that $\sigma^n(y)=z$. Then $\epsilon(y)\in Y^{n+1}$ satisfies
\begin{equation}
\sigma^{n+1}(\epsilon^n(y))=\zeta^n(\sigma^n(y))=\zeta^n(z)=0
\end{equation}
Hence $\epsilon(y) \in \mathrm{ker}(\sigma^{n+1})=\mathrm{Im}(\tau^{n+1})$. Thus the is an $x \in X^{n+1}$ such that $\tau^{n+1}(x)=\epsilon^n(y)$. It is easy to check that $x \in \mathrm{Ker}(\delta^{n+1})$ and the class $\bar{x}=x+\mathrm{Im}(\delta^n)\in \mathrm{H}^{n+1}(X)$ depends only in the class $\bar{z}$ of $z$ in $\mathrm{H}^n(Z)$. The connecting homomorphism sends $\bar{z}$ to $\bar{x}$.\\

 For the next two sections all the tensor products are over $k$ unless otherwise specified.

\section{Integrable derivations of degree $r$}
\begin{definition}(cf. $\cite[1.1]{Matsumura}$)
Let $A$ be a finite-dimensional $k$-algebra. A $\it{higher}$ $\it{derivation}$ $\underline{D}$ of $A$ is a sequence  $\underline{D}=(D_i)_{i\geq 0}$ of $k$-linear endomorphisms $D_i: A \to A$ such that $D_0=\mathrm{Id}$ and
 $D_n(ab)= \sum_{i+j=n}D_i(a)D_j(b)$ for all $n \geq 1$ and all $a, b \in A$. 
\end{definition}



 For a fixed positive integer $r$  we denote by $\mathrm{Aut_r}(A[[t]])$ the group of all $k[[t]]$-algebra automorphism of $A[[t]]$ which the induce the identity on $A[[t]] / t^{r}A[[t]]$. Clearly we have an inclusion $\mathrm{Aut}_r(A[[t]])\subseteq \mathrm{Aut}_1(A[[t]])$ for every $r\geq 1$.\\
 Following $\cite{Matsumura}$ any higher derivation $\underline{D}=(D_i)_{i\geq 0}$ of $A$ determines a unique automorphism $\alpha \in \mathrm{Aut}_r(A[[t]])$ satisfying $\alpha(a)=\sum_{i \geq 0} D_i(a)t^i$ for all $a\in A$ and vice versa. Note that any $k[[t]]$-ring endomorphism of $A[[t]]$ is determined by its restriction to $A$.
We denote by $\mathrm{Out_r}(A[[t]])$ the image of the canonical map $\varphi: \mathrm {Aut_r}(A[[t]])\to \mathrm{Out}(A[[t]])$  and by  $\mathrm{Der}(A)$  the set of derivations over $A$.

\begin{lemma}
\label{out}
Let $A$ be a finite-dimensional $k$-algebra. Let $r$ be a positive integer. Then $\mathrm{Out_r}(A[[t]])$ is the kernel of the canonical group homomorphism 
\begin{equation}
\psi:\mathrm{Out}(A[[t]]) \to \mathrm{Out}(A[[t]]/ (t^rA[[t]])).
\end{equation}
\end{lemma}
\begin{proof}
 Clearly $\mathrm{Out}_r(A[[t]])\subseteq \mathrm{Ker} (\psi)$. Let $\alpha$ be a representative of an element in the kernel of $\psi$. Then $\psi({\alpha})$ is given by conjugation with an invertible element $\bar{u}= u+t^{r}A[[t]]$ in $A[[t]]/t^{r}A[[t]]$ where $u \in  A[[t]]$. If we denote by $\overline{A[[t]]}=A[[t]]/t^{r}A[[t]]$, since $\bar{u}$ is invertible in $\overline{A[[t]]}$, we have $ \overline{A[[t]]}=\overline{A[[t]]} \bar{u}$.
  Then we can lift it to $A[[t]]=A[[t]]u+t^{r}A[[t]]$.
   By Nakayama's Lemma we have $A[[t]]=A[[t]]u$ hence $u$ is invertible. 
Consequently if we replace $\alpha$ by $\alpha$ composed with the conjugation given by $u^{-1}$ then the resulting automorphism is in the same class as $\alpha$ and it induces the identity on $A[[t]]/t^{r}A[[t]]$. 
\end{proof}

Slightly extending Matsumura $\cite{Matsumura}$ we have the following terminology:
\begin{definition}
Let $A$ be a finite-dimensional $k$-algebra and let $r$ be a positive integer. A derivation $D \in \mathrm{Der}(A)$ is called $\it{r}$-$\it{integrable}$ if there exists a higher derivation $\underline{D}=(D_i)_{i \geq 0}$ such that $D_0=\mathrm{Id}$, $D_i=0$ for $1\leq i\leq r-1$, and $D=D_r$. We  denote by $\mathrm{Der}_r(A)$ the set of $r$-integrable derivations of $A$ and by $\mathrm{Inn}_r(A)$ the set of $r$-integrable which are inner. 
\end{definition}

It is easy to check that these are abelian groups, using Proposition $\ref{cor1.2}$.

Note that for $r=1$ we have the usual notion of integrable derivation that is integrable derivations are $1$-integrable.
We recall from $\cite[1.5] {Matsumura}$:
 \begin{proposition}
\label{cor1.2}
The set of higher derivation is a group with the product defined on $(\ref{1.2})$. In particular if we let $\underline{D}=(D_i)_{i \geq 0}$ and $\underline{D'}=(D'_i)_{i \geq 0}$  be two higher derivations, then:
\begin{equation}
\label{1.2}
\underline{D}\ \circ \underline{D'}=\Big(\sum_{i=0}^n D_i\circ D'_{n-i}\Big )_{n \geq 0}\\
\end{equation}
\end{proposition}

\begin{proposition}
\label{1.3}
Let $A$ be a finite-dimensional $k$-algebra. Let $r$ be a positive integer, let $\alpha\in \mathrm{Aut_r}(A[[t]])$ and let $\mu: A[[t]] \to A[[t]]$ be the unique $k[[t]]$-linear map such that $\alpha(a)= a + t^r \mu(a)$ for all $a \in A[[t]]$. Then the following hold:
\begin{enumerate}[label=$($\alph*$)$]
\item The map $\bar{\mu}: A \cong  A[[t]]/tA[[t]] \to A\cong A[[t]]/tA[[t]]$ induced by $\mu$ is a derivation.
 In addition if $\alpha$  is an inner automorphism, then  $\bar{\mu}=[\bar{d},-]$ for some $\bar{d} \in A$; that is $\bar{\mu}$ is a inner derivation. 
\item The class of $\bar{\mu} \in \mathrm{HH}^1(A)$ depends only on the class of $\alpha \in \mathrm{Out}(A[[t]])$.
\end{enumerate}
\end{proposition}

\begin{proof}
 Let $a,b \in A[[t]]$, since $\alpha$ is an automorphism we have $\alpha(ab)=ab+t^r\mu(ab)$ is equal to $\alpha(a)\alpha(b)=ab+t^r\mu(a)+t^r\mu(b)+t^{2r}\mu(a)\mu(b)$ hence we obtain $\mu(ab)=a\mu(b)+\mu(a)b+t^r\mu(a)\mu(b)$. Reducing modulo $t^r$ we have
 \begin{equation}
 \mu(ab)=a \mu(b)+\mu(a)b
 \end{equation}
 hence $\bar{\mu}$ is a derivation on $A$.\\
  Now suppose that $\alpha$ is an inner automorphism induced by conjugation by an element $c\in (A[[t]])^{\times}$ that is $\alpha(a)=cac^{-1}$. Since $\alpha$ induces the identity on $A[[t]]/t^rA[[t]]$ then taking the projection of $\alpha$ in $A[[t]]/t^rA[[t]]$ we have $\bar{c}\bar{a}\bar{c}^{-1}=\bar{a}$, that is $\bar{c}\bar{a}=\bar{a}\bar{c}$ hence $\bar{c}\in Z(A[[t]]/t^rA[[t]])^{\times}$. Since the map $Z(A[[t]]) \to Z(A[[t]]/t^rA[[t]])$ is surjective then there is an element $z \in Z(A)^{\times}$ such that $\bar{z}=\bar{c}$ hence such that $cz^{-1} \in 1+ t^rA[[t]]$. So if we replace $c$ by $cz^{-1}$ we have $c=1+t^rd$ for some $d \in A[[t]]$. If we take an $a \in A[[t]]$ we have $cac^{-1}=\alpha(a)=a+t^r\mu(a)$ hence $ca=ac+t^r\mu(a)c$, that is $ [c,a]=t^r\mu(a)c$.
   Now if we replace $c$ by $1+t^rd$ and we divide by $t^r$ we obtain 
     \begin{equation}
  \label{1.4}
   [d,a]=\mu(a)+t^r\mu(a)d.
     \end{equation}
Consequently $[\bar{d},\bar{a}]=\bar{\mu}(\bar{a})$ hence the result.\\
  For the second part  we let $\alpha_1, \alpha_2$ be two representatives in $\mathrm{Out}(A[[t]])$ with induced derivations $\mu_1, \mu_2$. Since $\alpha_1\circ \alpha_2^{-1} \in \mathrm{Inn}(A[[t]])$ then using Proposition $\ref{cor1.2}$ and first part of the Proposition we have that $\mu_1-\mu_2$ $\in \mathrm{Inn}(A)$. Hence the result.
\end{proof}

An equivalent definition of $r$-integrable can be deduced from the following: let  $\alpha \in \mathrm{Aut}_r(A[[t]])$ and let $a=\sum_{i=0}^{\infty}a_it^i$. Then
$\alpha(a) = \sum_{i,n\geq 0}D_n(a_i) t^{i+n}= a+t^r\sum_{k\geq r}\sum_{\substack{n, i \geq 1, i+n=k}}^kD_n(a_{i})t^{r-i-n}$ since $D_i=0$ for $1\leq i\leq r-1$. 
Hence we can write $\alpha$ as $\alpha(a)=a+t^r \mu(a)$ where $\mu$ is an linear endomorphism of $A[[t]]$. From Proposition $\ref{1.3}$ the map $\bar{\mu}:A \to A$ induced by $\mu$ is a derivation over  $A$, in fact, $\bar{\mu}$ is exactly $D_r$.  Hence a derivation  $D$ on $A$ is $r$-integrable if there is an algebra automorphism  of $A[[t]]$, say $\alpha$, and a $k[[t]]$-linear endomorphism $\mu$ of $A[[t]]$ such that $\alpha (a)= a+t^r\mu(a)$ for all $a\in A[[t]]$ and such that $D$ is equal to the map $\bar{\mu}$ induced by $\mu$ on $A\cong A[[t]]/tA[[t]]$.


\begin{proposition}
Let $A$ be a finite-dimensional $k$-algebra and let $\alpha \in \mathrm{Aut}_1(A[[t]])$. Let $(D_i)_{i\geq0}$ be a higher derivation satisfying $\alpha(a)=\sum_{i \geq 0} D_i(a)t^i$ for $a \in A$. The map that  sends $\alpha$ to $\sum_{i \geq 0}D_it^i$ induces a group homomorphism $\phi:\mathrm{Aut}_1(A[[t]]) \to (\mathrm{End}_k(A)[[t]])^{\times}$.
\end{proposition}
\begin{proof}  
Let  $\beta \in \mathrm{Aut}_1(A[[t]])$. For $l \geq 0$ let $E_l \in \mathrm{End}_k(A)$ such that $\beta(a)=\sum_{l \geq 0} E_l(a)t^l$. For all $a \in A$ let  $\{e_j\}_{1 \leq j \leq n}$ be a $k$-basis of $A$. For every $i \geq 0$ define  $\mu_{ij}: A \to k$ such that  $D_i(a)=\sum_{j=1}^{n}\mu_{ij}(a)e_j$ where $a\in A$.
 On one side we have:
\begin{equation}
\begin{split}
(\beta \circ \alpha) (a)&= \beta\Big(\sum_{i \geq 0}D_i(a)t^i\Big) = \sum_{j=1}^{n}\beta\Big(\sum_{i \geq 0}\mu_{ij}(a)t^i  e_j\Big)\\
&=\sum_{j=1}^{n}\sum_{i \geq 0}\mu_{ij}(a)t^i \beta( e_j) = \sum_{l \geq 0} \sum_{i \geq 0} \sum_{j=1}^n \mu_{ij}(a) E_l(e_j)t^{i+l} 
\end{split}
\end{equation}
where the third equation holds since $\beta$ is an automorphism over $k[[t]]$. If we fix a degree $m \in \mathbb{N}$ we have

\begin{equation}
\begin{split}
\sum_{\substack{l,i \\ i+l=m}} \sum_{j=1}^n \mu_{ij}(a) E_l(e_j)t^{i+l}&=\sum_{\substack{l,i \\ i+l=m}} E_l(\sum_{j=1}^n \mu_{ij}(a)e_j)t^{m} \\
&=\sum_{\substack{l,i \\ i+l=m}}  E_l(D_i(a))t^{m} 
\end{split}
\end{equation}
Hence  $\phi(\beta \circ \alpha)$ in degree $m$ is equal to $ \sum_{\substack{i,l\geq 0 \\ i+l=m}} E_l\circ D_i t^m$. This is clearly equal to the coefficient at $t^m$ of $\phi(\beta)\phi(\alpha)$.
 \end{proof}

\begin{definition}
Let $A$ be a finite-dimensional $k$-algebra. Let $r$ be a positive integer then by $\mathrm{HH}^1_r(A)$ we denote the quotient $\mathrm{Der}_r(A)/ \mathrm{Inn}_r(A)$. 
\end{definition}
Clearly $\mathrm{HH}^1_r(A)$ can be identified with a subgroup of $\mathrm{HH}^1(A)$.

\begin{proposition}
\label{1.4.1}
Let $A$ be a finite-dimensional algebra over $k$. Let $r$ be a positive integer and let $\alpha \in \mathrm{Aut_r}(A[[t]])$ . Let $\mu$ the unique $k[[t]]$-linear map on $A[[t]]$ such that $\alpha(a)=a + t^r \mu(a)$ for all $a \in A[[t]]$. We denote by $\bar{\mu}$ the derivation induced on $A$ by $\mu$. 
\begin{enumerate}[label=$($\alph*$)$]
\item The derivation $\bar{\mu}$ is inner if and only if $\alpha$ induces an inner automorphism in $A[[t]]/t^{r+1}A[[t]]$.
\item We have the following short exact sequence of groups:
\begin{center}
      \begin{tikzpicture}[node distance=3 cm]
  \node (P2) {$1$};
  \node (P1) [right of=P2] [node distance=2 cm]{$ \mathrm{Out_{r+1}}(A[[t]])  $};
  \node (P0) [right of=P1] {$  \mathrm{Out_{r}}(A[[t]]) $};
  \node (A) [right of=P0] [node distance=2.5 cm] {$ \mathrm{HH^1_r}(A)$};
    \node (A0) [right of=A][node distance=1.5 cm] {$1$};
 \draw[->] (P2) to node  {$$} (P1);
  \draw[->] (P1) to node {$$} (P0);
 \draw[->] (P0) to node   {$$} (A);
\draw[->] (A) to node  {$$} (A0);
\end{tikzpicture}
\end{center}

\end{enumerate}
\end{proposition}

\begin{proof}
Let assume  that $\bar{\mu}$ is inner derivation so $\bar{\mu}=[\bar{d},-]$ for some $d \in A[[t]]$. We can take $c=1+t^rd$ as in the proof of Proposition $\ref{1.3}$. Then from Equation $\ref{1.4}$ we can choose $\tau(a)=-\mu(a)d$ so we have $[d,a]=\mu(a)-t^r\tau(a)$ and since $c=1+t^rd$ then
\begin{equation}
[c,a]=[1+t^rd,a]=t^r[d,a]
\end{equation} 
So $[c,a]=t^r[d,a]=t^r\mu(a)-t^{2r}\tau(a)$. Hence $t^r\mu(a)=[c,a]+t^{2r}\tau(a)$. Consequently  $cac^{-1}=a+t^r\mu(a)c^{-1}-t^{2r}\tau(a)c^{-1}$. Using the fact that $\alpha(a)= a+t^{r}\mu(a)$ it follows that $\alpha(a)-cac^{-1}=t^{r}\mu(a)(1-c^{-1})+t^{2r}\tau(a)c^{-1}$. Since $c$ belongs to $1+t^{r}A[[t]]$, we have $c^{-1} \in 1+t^r A[[t]]$ hence $1-c^{-1} \in t^{r}A[[t]]$. This shows that $\alpha(a)-cac^{-1} \in t^{2r}A[[t]]\subset t^{r+1}A[[t]]$. Consequently $\alpha$ induces an inner automorphism on $A[[t]]/t^{r+1}A[[t]]$.\\
 Conversely, suppose that $\alpha$ acts as an inner automorphism on $A[[t]]/t^{r+1}A[[t]]$. Using the same argument as in Lemma $\ref{out}$ we may assume that $\alpha$ acts as identity on $A[[t]]/t^{r+1}A[[t]]$ hence it  induces an inner derivation on $A[[t]]/t^{r+1}A[[t]]$. Hence we can assume $\alpha$ such that $\alpha \in \mathrm{Aut}_{r+1}(A[[t]])$ . Hence $\alpha(a)=a+t^{r+1}\mu'(a)$ for some $\mu'(a)\in A[[t]]$, which gives the equality $\mu(a)=t\mu'(a)$. Consequently we have that $\mu$  induces the zero map on $A$.\\
For the second part let $\beta\in \mathrm{Aut_r}(A)$ such that $\beta(a)= a+t^r\nu(a)$ for all $a\in A[[t]]$ and for some linear morphism $\nu$ on $A[[t]]$. From Proposition $\ref{cor1.2}$ and Proposition $\ref{1.3}$ we have that  the class determined by $\beta\circ\alpha$ in $\mathrm{HH^1}(A)$ is the class determined by $\bar{\mu}+\bar{\nu}$. 
\end{proof}

A way to understand the action of the $p$-power map on the integrable derivations is by studying it on $\mathrm{Aut_1}(A[[t]])$ and then using the homomorphism $\phi:\mathrm{Aut_1}(A[[t]]) \to (\mathrm{End_k}(A)[[t]])^{\times}$. 
\begin{proposition}
\label{propbin}
Let $\underline{D}$ be a higher derivation and let $l, n$ be positive integers. The term at $t^l$ in
$\Big(\sum_{i\geq 0}D_i t^i\Big)^n$ is equal to
\begin{equation}
\label{bin}
\sum_{c=1}^l \binom{n}{c} \sum_{\substack{ i_1, \dots i_c \geq 1 \\ i_1+\dots +i_c=l}} \prod_{j=1}^{c}D_{i_j}
\end{equation}
\end{proposition}

\begin{proof}
The term at $t^l$ in  $\Big(\sum_{i\geq 0}D_i t^i\Big)^n$  is given by 
\begin{equation}
\label{bin1}
\sum_{ \substack{ i_1, \dots, i_n\geq 0 \\ i_1+ \dots +i_n =l}} \prod_{j=1}^{n} D_{i_j}.
\end{equation}
 Let $c$ be a positive integer. Then for each $c$-tuple $(i'_1, i'_2, \dots, i'_c)$ which has non-zero components and such that $\sum_{j=1}^c i'_j=l$, there are $ \binom {n} {c}$ different $n$-tuples $(i_1, i_2, \dots ,i_n)$ which have the $c$ non-zero components of the $c$-tuple $(i'_1, i'_2, \dots, i'_c)$ and rest equal to zero. Since $D_0=\mathrm{Id}$ then $ \prod_{j=1}^n D_{i_j} = \prod_{j=1}^c D_{i'_j}$. For a fixed $c$ the Equation $(\ref{bin1})$ is given by $\binom{n}{c} \sum_{\substack{ i_1, \dots i_c \geq 1 \\ i_1+\dots +i_c=l}} \prod_{j=1}^{c}D_{i_j}$. If we sum over all $c$ we have the result.
\end{proof}

\begin{corollary}
\label{P1} 
Let $A$ be a finite-dimensional $k$-algebra and let $\alpha \in \mathrm{Aut_r}(A[[t]])$  for some positive integer $r$. Then $\alpha^p \in \mathrm{Aut}_{rp}(A[[t]])$. The $p$-power map sends $\mathrm{HH}^1_r(A)$ to $\mathrm{HH}^1_{rp}(A)$, and $\mathrm{Out}_r(A[[t]])$ to $\mathrm{Out}_{rp}(A[[t]])$ and we have a commutative diagram

 \begin{center}
 \begin{minipage}{0.3 \textwidth}
      \begin{tikzpicture}[node distance=3.5cm, auto]
  \node (HB1) {$\mathrm{Out}_r(A[[t]])$};
  \node (HB2) [below of=HB1] [node distance=2 cm] {$\mathrm{HH}_{r}^1(A)$};
  \node (HA1) [right of=HB1] {$\mathrm{Out}_{rp}(A[[t]])$};
  \node (HA2) [right of=HB2] {$\mathrm{HH}_{rp}^1(A)$};
  \draw[->] (HB1) to node {$(\ )^p$} (HA1);
  \draw[->] (HA1) to node {$$} (HA2);
 \draw[->] (HB1) to node   {$$} (HB2);
\draw[->] (HB2) to node  {$[p]$} (HA2);
\end{tikzpicture}
\end{minipage}
\end{center}

where the vertical maps are from Proposition $\ref{1.4.1}\ (b)$, $(\ )^p$ is the $p$-fold composition and $[p]$ is the $p$-power map.
\end{corollary}

\begin{proof}
Let $\alpha \in \mathrm{Aut_r}(A[[t]])$ and let $D_r$ the derivation in $\mathrm{Der}_r(A)$. Let $\underline{D'}$ be the higher derivation associated to $\alpha^p$. 
 Using  Proposition $\ref{propbin}$, in degree $l \leq p-1$ we have:
 \begin{equation}
 \sum_{c=1}^l \binom{p}{c} \sum_{\substack{ i_1, \dots i_c \geq 1 \\ i_1+\dots +i_c=l}} \prod_{j=1}^{c}D_{i_j} t^l=0
\end{equation}
 since the binomial coefficient give us multiples of $p$. For $l \geq p$
 \begin{equation}
 \sum_{c=1}^l \binom{p}{c} \sum_{\substack{ i_1, \dots i_c \geq 1 \\ i_1+\dots +i_c=l}} \prod_{j=1}^{c}D_{i_j}t^l = \sum_{\substack{ i_1, \dots i_p \geq 1 \\ i_1+\dots +i_p=l}} \prod_{j=1}^{c}D_{i_j} t^l
\end{equation}
Now we know that each $D_i$ is zero for $i=1,\dots ,r-1$ so in order to have an element different from zero we should impose that each $i_j$ be at least $r$. Therefore the sum 
$i_1+\dots+ i_p=rp$ that is $l=rp$ hence the first non-zero coefficient  is $D_r^p$. Consequently the diagram commutes. 
\end{proof}

\section{A cohomological interpretation of $r$-integrable derivations}
Integrable derivation can also being interpreted  using a cohomological point of view.
Starting from the short exact sequence of $A[[t]]$-$A[[t]]$-bimodules:
\begin{center}
 \begin{tikzpicture}[node distance=3 cm,auto]
  \node (P2) {$0$};
  \node (P1) [right of=P2][node distance=2 cm] {$A[[t]]$};
  \node (P0) [right of=P1] {$A[[t]]$};
  \node (A) [right of=P0] {$ A[[t]]/t^rA[[t]]$};
    \node (A0) [right of=A][node distance=2 cm] {$0$};
 \draw[->] (P2) to node  {$$} (P1);
  \draw[->] (P1) to node {$t^r$} (P0);
 \draw[->] (P0) to node   {$$} (A);
\draw[->] (A) to node  {$$} (A0);
\end{tikzpicture}
\end{center}

after dividing by $tA[[t]]$ and  twisting on the right by the automorphism $\alpha \in \mathrm{Aut}_r(A[[t]])$  we obtain the short exact sequence:

\begin{center}
 \begin{tikzpicture}[node distance=3 cm,auto]
  \node (P2) {$0$};
  \node (P1) [right of=P2] [node distance=2 cm] {$A[[t]]/tA[[t]]$};
  \node (P0) [right of=P1] {$(A[[t]]/t^{r+1}A[[t]])_{\alpha}$};
  \node (A) [right of=P0] {$ A[[t]]/t^rA[[t]]$};
    \node (A0) [right of=A] [node distance=2 cm] {$0$};
 \draw[->] (P2) to node  {$$} (P1);
  \draw[->] (P1) to node {$t^r$} (P0);
 \draw[->] (P0) to node   {$$} (A);
\draw[->] (A) to node  {$$} (A0);
\end{tikzpicture}
\end{center}

since $\alpha$ induces the identity on $A[[t]]/t^rA[[t]]$ hence also on $A[[t]]/tA[[t]]$.

The following proposition is an adaptation of $\cite[4.1]{Markus}$ to the situation under consideration.
\begin{proposition}
\label{con}
 Let $A$ be a finite-dimensional algebra over $k$. Set $\hat{A}=A[[t]]$ and set $\hat{A}^e=\hat{A}\otimes_{k[[t]]}\hat{A}^{op}$.  Let $\alpha \in \mathrm{Aut_r}(\hat{A})$. Let $r$ a positive integer and let $\mu : A \to A$ be the unique linear map satisfying  $\alpha(a)=a+t^r\mu(a)$.  Let $P$ be a projective resolution of $\hat{A}$ as $\hat{A}^e$-module. Applying the functor $\mathrm{Hom}_{\hat A}(P,-)$ to the  exact sequence of $\hat{A}^e$-modules

\begin{center}
 \begin{tikzpicture}[node distance=3 cm,auto]
  \node (P2) {$0$};
  \node (P1) [right of=P2][node distance=2 cm] {$\hat{A}/t\hat{A}$};
  \node (P0) [right of=P1] {$(\hat{A}/t^{r+1}\hat{A})_{\alpha}$};
  \node (A) [right of=P0] {$ \hat{A}/t^r\hat{A}$};
    \node (A0) [right of=A] [node distance= 2 cm] {$0$};
 \draw[->] (P2) to node  {$$} (P1);
  \draw[->] (P1) to node {$t^r$} (P0);
 \draw[->] (P0) to node   {$$} (A);
\draw[->] (A) to node  {$$} (A0);
\end{tikzpicture}
\end{center}
yields a short exact sequence of cochain complexes

\begin{center}
 \begin{tikzpicture}[node distance=3.8 cm,auto]
  \node (P2) {$0$};
  \node (P1) [right of=P2] [node distance=2 cm] {$\mathrm{Hom}_{\hat{A}^{e}}(P,A)$};
  \node (P0) [right of=P1]  {$\mathrm{Hom}_{\hat{A}^{e}}(P,(\hat{A}/t^{r+1}\hat{A})_{\alpha})$};
  \node (A) [right of=P0]{$\mathrm{Hom}_{\hat{A}^{e}}(P, \hat{A}/t^r\hat{A})$};
    \node (A0) [right of=A] [node distance=2 cm]  {$0$};
 \draw[->] (P2) to node  {$$} (P1);
  \draw[->] (P1) to node {$t^r$} (P0);
 \draw[->] (P0) to node   {$$} (A);
\draw[->] (A) to node  {$$} (A0);
\end{tikzpicture}
\end{center}
The first non trivial connecting homomorphism can be identified with a map
\begin{equation}
\mathrm{End}_{\hat{A}^e}(\hat{A}/t^r\hat{A})\to \mathrm{HH^1}(A)
\end{equation}
and this map sends $\mathrm{Id}_{\hat{A}/t^r\hat{A}}$ to the class of the derivation induced by $\mu$ on $A$.
\end{proposition}

\begin{proof}
We take as a projective resolution the bar resolution $P$ of $\hat{A}$ where the tensor products are over $k[[t]]$:

\begin{center}
 \begin{tikzpicture}[node distance=2.5 cm,auto]
  \node (P2) {$\dots$};
  \node (P1) [right of=P2] [node distance=2 cm] {$\hat{A}^{\otimes n+2}$};
  \node (P0) [right of=P1]  {$\hat{A}^{\otimes n+1}$};
  \node (A) [right of=P0]{$\dots$};
 \draw[->] (P2) to node  {$$} (P1);
  \draw[->] (P1) to node {$\delta_n$} (P0);
 \draw[->] (P0) to node   {$$} (A);
\end{tikzpicture}
\end{center}

which is given by $\delta_n(a_0\otimes \dots \otimes a_{n+1})=\sum_{i=0}^{n}(-1)^i a_0\otimes \dots \otimes a_ia_{i+1}\otimes \dots \otimes a_{n+1}$. The last non-zero differential is the map $\delta_1:\hat{A}^{\otimes 3} \to \hat{A}^{\otimes 2} $  which sends
$a\otimes b\otimes c$ to $ab\otimes c -a\otimes bc$ 
for $a,b,c \in \hat{A}$. 
We have the following identifications:
\begin{equation}
\begin{split}
\mathrm{H}^0(\mathrm{Hom}_{\hat{A}^e}(P,\hat{A}/t^r\hat{A}))&=\mathrm{HH}^0(\hat{A},\hat{A}/ t^r\hat{A}) \\
&\cong\mathrm{HH}^0(\hat{A}/t^r\hat{A})= \mathrm{End}_{\hat{A}^e}(\hat{A}/t^r\hat{A})
\end{split}
\end{equation}
The identity map in $\mathrm{End}_{\hat{A}^e}(\hat{A}/t^r\hat{A})$ corresponds to the homomorphism 
\begin{equation} 
\begin{split}
\zeta:\hat{A}\otimes_{k[[t]]} \hat{A} &\to \hat{A}/t^r\hat{A}\\
a\otimes b &\mapsto \zeta(a\otimes b)=ab+t^r\hat{A}
\end{split}
\end{equation}
 for all $a,b \in A [[t]]$. This lifts to an $\hat{A}^e$-homomorphism 
 \begin{equation}
 \begin{split}
 \bar{\zeta}: \hat{A}\otimes_{k[[t]]} \hat{A} &\to (\hat{A}/t^{r+1}\hat{A})_{\alpha}\\
 a\otimes b &\mapsto \bar{\zeta}(a \otimes b)= a\alpha(b)+t^{r+1}\hat{A}
 \end{split}
 \end{equation}
  for $a, b \in \hat{A}$ since $\alpha$ induces the identity on $\hat{A}/t^r\hat{A}$.\\
Since $\bar{\zeta} \in \mathrm{Hom}_{\hat{A}}(\hat{A}\otimes \hat{A}, (\hat{A}/t^{r+1}\hat{A})_{\alpha})$  we need to apply the first non-zero differential
\begin{equation}
 \epsilon: \mathrm{Hom}_{\hat{A}^e}(\hat{A}^{\otimes 2}, (\hat{A}/t^{r+1}\hat{A})_{\alpha}) \to \mathrm{Hom}_{\hat{A}^e}(\hat{A}^{\otimes 3}, (\hat{A}/t^{r+1}\hat{A})_{\alpha})
 \end{equation} 
  which is given by composing with $-\delta_1$.
Hence in  $\hat{A}/t^{r+1}\hat{A}$ we have:
\begin{equation}
\begin{split}
(-\bar{\zeta}\circ \delta_1)(a\otimes b\otimes c)&= -\bar{\zeta}(ab \otimes c +a \otimes bc)= -ab\alpha(c)+a\alpha(bc)\\
&=a(\alpha(b)-b)\alpha(c)=t^ra\mu(b)\alpha(c).
\end{split}
\end{equation}
for all $a,b,c \in \hat{A}$. We observe that $t^ra\mu(b)\alpha(c)+t^{r+1}\hat{A} \in \hat{A}/t^{r+1}\hat{A}$ is the image, under $t^r:\hat{A}/t\hat{A} \to (\hat{A}/t^{r+1}\hat{A})_{\alpha}$, of the map $ \psi: \hat{A}^{\otimes 3} \to \hat{A}/t\hat{A}$, that is we have the following commutative diagram:
 \begin{center}
 \begin{minipage}{0.3 \textwidth}
      \begin{tikzpicture}[node distance=2.5cm, auto]
  \node (HB1) {$\hat{A/t\hat{A}}$};
  \node (HA1) [right of=HB1] {$(\hat{A/t^{r+1}\hat{A}})_{\alpha}$};
  \node (HA2) [above of=HA1] {$\hat{A}^{\otimes 3}$};
  \draw[->] (HB1) to node {$t^r$} (HA1);
  \draw[->] (HA2) to node {$-\zeta\circ \delta$} (HA1);
 \draw[->] (HA2) to node [swap]  {$\psi$} (HB1);
\end{tikzpicture}
\end{minipage}
\end{center}
  where $\psi$  sends $a\otimes b\otimes c$ to $a\mu(b)\alpha(c)+t\hat{A}$ which is equal to $a\mu(b)c+t\hat{A}$ since $\alpha(c)-c \in t^r\hat{A}\subseteq t\hat{A}$. Consequently $\psi$ induces a map $\bar{\psi}:\hat{A}^{\otimes 3} \to A$ which sends $\bar{a} \otimes \bar{b} \otimes \bar{c}$ to $\bar{a}\bar{\mu}(\bar{b})\bar{c}$ that can be restricted to the map $\bar{\psi}: A\to A$ that sends $\bar{b}$ to $\mu(\bar{b})$. Using $(\ref{1eq})$  the result follows.


\end{proof}

\section{Proof of Theorem $\ref{mainthm}$}
The proof of Theorem $\ref{mainthm}$ requires the following result, which is a variation of $\cite[5.1]{Markus}$:
 \begin{theorem}
 \label{thm1}
Let $A, B$ be finite-dimensional selfinjective $k$-algebras with separable semisimple quotients.  Let $r$ be a positive integer and let $M$, $N$ be an $A$-$B$-bimodule, $B$-$A$ bimodule, respectively, inducing a stable equivalence of Morita type between $A$ and $B$. Then for any $\alpha \in \mathrm{Aut}_r(A[[t]])$ there is $\beta \in \mathrm{Aut}_r(B[[t]])$ such that
 $_{\alpha^{-1}}M[[t]] \cong M[[t]]_{\beta}$ as $A[[t]]$-$B[[t]]$-bimodules. This correspondence induce a group isomorphism
  $\mathrm{Out}_r(A[[t]]) \cong \mathrm{Out}_r(B[[t]])$ making the following diagram commutative:

 \begin{center}
 \begin{minipage}{0.3 \textwidth}
      \begin{tikzpicture}[node distance=3.5cm, auto]
  \node (HB1) {$\mathrm{Out}_r(A[[t]])$};
  \node (HB2) [below of=HB1] [node distance=2cm] {$HH^1_r(A)$};
  \node (HA1) [right of=HB1] {$\mathrm{Out}_r(B[[t]])$};
  \node (HA2) [right of=HB2] {$HH^1_r(B)$};
  \draw[->] (HB1) to node {$\cong$} (HA1);
  \draw[->] (HA1) to node {$$} (HA2);
 \draw[->] (HB1) to node   {$$} (HB2);
\draw[->] (HB2) to node  {$\cong$} (HA2);
\end{tikzpicture}
\end{minipage}
\end{center}

where the vertical maps are from Proposition $\ref{1.4.1}$ and the lower horizontal isomorphism is induced by the functor $N\otimes_A - \otimes_A M$
 \end{theorem}

\begin{proof}
 By the Lemma $\cite[4.2]{Markus}$ we have that the upper horizontal map is a group isomorphism. 
Let $\alpha \in \mathrm{Aut}_r(A[[t]])$, $\beta\in \mathrm{Aut}_r(B[[t]])$ such that $_{\alpha^{-1}}M[[t]] \cong M[[t]]_{\beta}$ as $A[[t]]$-$B[[t]]$-bimodules. We also have that $\alpha$ is such that $\alpha(a)=a+t^r\mu(a)$ for all $a\in A[[t]]$ and $\beta$ such that $\beta(b)=b+t^r\nu(b)$ for all $b \in B[[t]]$ for some $k[[t]]$-linear endomorphisms $\mu, \nu$. We denote by $\bar{\mu}$ and $\bar{\nu}$ the classes in $\mathrm{HH}^1_r(A)$ and $\mathrm{HH}^1_r(B)$ respectively determined by the canonical group homomorphism $\mathrm{Out}_r(A[[t]])\to \mathrm{HH}^1(A)$ and $\mathrm{Out}_r(B[[t]])\to \mathrm{HH}^1(B)$. Set $\hat{M}=M[[t]]$. By the assumptions, tensoring by $M$ yields a stable equivalence of Morita type between $A$ and $B$. In particular we have:
\begin{equation}
\mathrm{HH}^1(A)\cong \mathrm{Ext}^1_{A\otimes_k B^{op}}(M,M) \cong \mathrm{HH}^1(B)
\end{equation}
induced by the functors $-\otimes_{A} M$ and $M \otimes_{B} -$.
In addition since $B[[t]]$ is isomorphic to  $\hat{N} \otimes_{A[[t]]} \hat{M}$ in the relatively $k[[t]]$-stable category of $B[[t]] \otimes_{k[[t]]} B[[t]]^{op}$-modules, it follows that the isomorphism 
\begin{equation}
\mathrm{HH}^1(A) \cong \mathrm{HH}^1(B)
\end{equation}
 given by the composition of the two previous isomorphisms is induced by the functor $N \otimes_{A}\ -\ \otimes_{A} M$. 
 The functors $M\otimes_B -$, $-\otimes_A M$ also induce algebra homomorphisms
\begin{equation}
\mathrm{End}_{A^e}(A) \to \mathrm{End}_{A\otimes B^{op}}(M)\leftarrow \mathrm{End}_{B^e}(B)
\end{equation}
where $A^e=A\otimes_k A^{op}$ and similarly for $B^e$. Tensoring the following two exact sequence
 \begin{center}
 \begin{tikzpicture}[node distance=3.5cm,auto]
  \node (P2) {$0$};
  \node (P1) [right of=P2][node distance=1.5cm] {$A$};
  \node (P0) [right of=P1] {$(A[[t]]/t^{r+1}A[[t]])_{\alpha}$};
  \node (A) [right of=P0] {$ A[[t]]/t^rA[[t]]$};
  \node (A0) [right of=A] [node distance=2cm] {$0$};
 \draw[->] (P2) to node  {$$} (P1);
  \draw[->] (P1) to node {$$} (P0);
 \draw[->] (P0) to node   {$$} (A);
\draw[->] (A) to node  {$$} (A0);
\end{tikzpicture}
\end{center}
 and 
 
  \begin{center}
 \begin{tikzpicture}[node distance=3.5 cm,auto]
  \node (P2) {$0$};
  \node (P1) [right of=P2] [node distance=1.5cm] {$B$};
  \node (P0) [right of=P1] {$(B[[t]]/t^{r+1}B[[t]])_{\alpha}$};
  \node (A) [right of=P0] {$ B[[t]]/t^rB[[t]]$};
    \node (A0) [right of=A][node distance=2cm] {$0$};
 \draw[->] (P2) to node  {$$} (P1);
  \draw[->] (P1) to node {$$} (P0);
 \draw[->] (P0) to node   {$$} (A);
\draw[->] (A) to node  {$$} (A0);
\end{tikzpicture}
\end{center}
 by $-\otimes_{A[[t]]} \hat{M}$ and $\hat{M} \otimes_{B[[t]]} -$ yields short exact sequences of the form
  \begin{center}
 \begin{tikzpicture}[node distance=2.5 cm,auto]
  \node (P2) {$0$};
 \node (P1) [right of=P2] [node distance=2cm] {$M$};
  \node (P0) [right of=P1] {$_{\alpha^{-1}}(M[[t]]/t^{r+1}M[[t]])$};
  \node (A) [right of=P0] {$ M$};
  \node (A0) [right of=A] [node distance=2cm] {$0$};
 \draw[->] (P2) to node  {$$} (P1);
 \draw[->] (P1) to node {$$} (P0);
 \draw[->] (P0) to node   {$$} (A);
\draw[->] (A) to node  {$$} (A0);
\end{tikzpicture}
\end{center}
 
   \begin{center}
    \begin{tikzpicture}[node distance=2.5 cm,auto]
  \node (P2) {$0$};
  \node (P1) [right of=P2] [node distance=2cm] {$M$};
  \node (P0) [right of=P1] {$(M[[t]]/t^{r+1}M[[t]])_{\beta}$};
  \node (A) [right of=P0] {$ M$};
    \node (A0) [right of=A] [node distance=2cm] {$0$};
 \draw[->] (P2) to node  {$$} (P1);
  \draw[->] (P1) to node {$$} (P0);
 \draw[->] (P0) to node   {$$} (A);
 \draw[->] (A) to node  {$$} (A0);
 \end{tikzpicture}
 \end{center}
 
 By the naturality properties of the connecting homomorphism and from the description of $\bar{\mu}$, $\bar{\nu}$ in Proposition $\ref{con}$
 the image of $\bar{\mu}\otimes \mathrm{Id}_M$ and $\mathrm{Id}_M\otimes \bar{\nu}$ in $\mathrm{Ext}^1_{A\otimes_k B^{op}}(M,M)$
 are equal to the images of $\mathrm{Id}_{\hat{M}}$ under the two connecting homomorphisms 
 \begin{equation}
 \mathrm{End}_{A\otimes_{k}B^{op}}(\hat{M}) \to \mathrm{Ext}^1_{A\otimes_k B^{op}}(M,M)
 \end{equation}
 obtained after applying the functor $\mathrm{Hom}_{A[[t]]\otimes B[[t]]^{op}}(\hat{M}, -)$ to the short exact sequences using the same identification used in Proposition $\ref{con}$.
 By the Lemma $\cite[4.3]{Markus}$
 the two exact sequences are equivalent, consequently the connecting homomorphism are equal. Hence the two images of $\mathrm{Id}_{M}$ coincide. This shows that the group isomorphism $\mathrm{HH}^1_r(B)\cong \mathrm{HH}^1_r(A)$ induced by $\mathrm{Out}_r(B[[t]])\cong \mathrm{Out}_r(A[[t]])$ is equal to the one determined by the functor $N\otimes_{A}- \otimes_{A} M$. Hence the result. 
\end{proof}

\begin{proof}[Proof of Theorem $\ref{mainthm} $]
We show first that the following diagram commutes:

 \begin{center}
 \begin{minipage}{0.3 \textwidth}
      \begin{tikzpicture}[node distance=3.5cm, auto]
  \node (HB1) {$\mathrm{Out}_r(A[[t]])$};
  \node (HB2) [below of=HB1] [node distance=2 cm] {$\mathrm{Out}_{rp}(A[[t]])$};
  \node (HA1) [right of=HB1] {$\mathrm{Out}_{r}(B[[t]])$};
  \node (HA2) [right of=HB2] {$\mathrm{Out}_{rp}(B[[t]])$};
  \draw[->] (HB1) to node {$\cong$} (HA1);
  \draw[->] (HA1) to node {$(\ )^p$} (HA2);
 \draw[->] (HB1) to node   {$(\ )^p$} (HB2);
\draw[->] (HB2) to node  {$\cong$} (HA2);
\end{tikzpicture}
\end{minipage}
\end{center}
where the horizontal maps are from Theorem $\ref{thm1}$ and the vertical maps are $p$-fold compositions. Let $\alpha \in \mathrm{Aut}_r(A[[t]])$ and $\beta \in \mathrm{Aut}_r(B[[t]])$  such that $_{\alpha^{-1}}M[[t]] \cong M[[t]]_\beta$. Let $\mu, \nu$ be the unique linear maps on $A[[t]]$ such that $\alpha(a)= a+ t^r\mu(a)$ and $\beta(b)= b+t^r\nu(b)$ respectively. By Corollary $\ref{P1}$ we have $\alpha^p \in \mathrm{Aut}_{rp}(A[[t]])$, $\beta^p \in \mathrm{Aut}_{rp}(A[[t]])$ and also that the maps $\bar{\mu}$, $\bar{\nu}$, induced by $\mu, \nu$ on $A$, is  sent under the $p$-power map to $\bar{\mu}^p$ and $\bar{\nu}^p$ respectively. Hence we have the commutativity of the diagram above since $_{\alpha^{-p}}M[[t]] \cong M[[t]]_{\beta^p}$. Using the commutative diagram above and Theorem $\ref{thm1}$ we have that the class of $\bar{\mu}^p$ is sent though the isomorphism defined in Theorem \ref{mainthm} to the class of $\bar{\nu}^p$. Hence we have the commutativity of the diagram of the Theorem $\ref{mainthm}$.

\end{proof}

\section{Example}

The purpose of the following example is to show that $p$-power maps do not commute in general with transfer maps in the Hochschild cohomology of symmetric algebras.\\
Let  $H=\{1,(123), (132)\} \cong C_3 \leq S_3$ and $M= kS_3$ considered as a $kS_3$-$kC_3$ bimodule. By $\langle - , - \rangle$ we mean the standard bilinear form for the group algebra $kH$. We choose $\{1, t=(12)\}$ as set of representatives of $S_3/H$. We note that $M$ is finitely generated and projective as a right $kC_3$-module, since $[G:H]=2$, so there exist $x_i\in \mathrm{Hom}_{kC_3}(kS_3,kC_3)$ with $1\leq i \leq 2$ such that for any $x\in M$, $x=\sum_i x_i\varphi_i(x)$. Explicitly:
\begin{equation}
\begin{split}
&\varphi_1(1)=1, \varphi_1((123))=(123), \\
&\varphi_1((132))=(132), \varphi_1(g)=0
\end{split}
\end{equation}
 for every other $g \in G$.
Similarly we define:
\begin{equation}
\varphi_t(12)=1, \varphi_t((13))=(132), \varphi_t((23))=(123), \varphi_t(g)=0
\end{equation}
 for every other $g \in G$.
 Since $C_3$ is commutative then $\mathrm{HH}^1(kC_3)= \mathrm{Der}_k(kC_3)$ which is generated by $\{f_0,f_1,f_2\}$ such that $f_0((123))=1$, $f_1((123))=(123)$ and $f_2((123))=(132)$. In this case the explicit formula of the transfer map by $\cite [2.5] {KLZ}$ is given by:
\begin{equation}            
\begin{split}
\mathrm{tr}^M(f) &=\sum_{h \in H}\Big\langle h^{-1},f(\varphi_1(a)) \Big\rangle h +\Big\langle h^{-1}t,f(\varphi_t(a)) \Big\rangle th+\\
&\Big\langle h^{-1},f(\varphi_1(at)) \Big\rangle ht+ \Big\langle h^{-1},f(\varphi_t(at)) \Big\rangle tht
\end{split}
\end{equation}
where $f \in Der_k(kC_3)$. 
In particular for $a=(123)$ we have:
\begin{equation}
\begin{split}
\mathrm{tr}^M(f_0)((123))&= \sum_{h\in H}\Big(\Big\langle h^{-1},f_0((123))\Big\rangle + \Big\langle h,f_0((132))\Big\rangle\Big) h \\
&= \sum_{h \in H} \Big(\Big\langle h^{-1},1\Big\rangle + \Big\langle h,-(123)\Big\rangle\Big) h = 1- (132)
\end{split}
\end{equation}
similarly we have:
\begin{equation}
\mathrm{tr}^M(f_0)(132)= 1- (123).
\end{equation}
We can  note now that $\mathrm{tr}^M(f_0^{[3]})=0$ since $f_0^{[3]}=0$, so $\mathrm{tr}^M(f_0^{[3]})(132)=0$.
On the other hand $\mathrm{tr}^M(f_0)^{[3]}((132))= \mathrm{tr}^M(f_0)\circ \mathrm{tr}^M(f_0)(1-(123))= \mathrm{tr}^M(f_0)(-1+(132))= 1-(123)$. 
Since the transfer maps send elements on $\mathrm{HH}^1(B)$ to elements $\mathrm{HH}^1(A)$
 it should exists a inner derivation in $S_3$ which sends $(132)$ to
 $1-(123)$ if we require the commutativity of the diagram.
But there is no element in $ a\in kS_3$ such that $[a,(132)]=1$. Hence in this case the $p$-power map does not commute
with the transfer map.\\
\begin{remark}
This shows that the $p$-power map cannot be expressed in terms of the $BV$-operator, as this is invariant under transfer maps, by $\cite[10.7]{KLZ}$.
\end{remark}


\end{document}